\documentclass[preprint,onefignum,onetabnum]{siamonline190516}

 \usepackage{mathrsfs}
\usepackage{multirow,mfirstuc}
\usepackage{subfigure}

\newtheorem{assumption}[theorem]{Assumption}
\newcommand{\Acal}{{\cal A}}

\newcommand{\Ical}{{\cal I}}

\newcommand{\Lcal}{{\cal L}}

\newcommand{\Scal}{{\cal S}}

\newcommand{\sign}{\text{sign}}   
     
\usepackage{lipsum}
\usepackage{amsfonts}
\usepackage{graphicx}
\usepackage{epstopdf}
\usepackage{algorithmic}
\ifpdf
  \DeclareGraphicsExtensions{.eps,.pdf,.png,.jpg}
\else
  \DeclareGraphicsExtensions{.eps}
\fi

\usepackage{enumitem}
\setlist[enumerate]{leftmargin=.5in}
\setlist[itemize]{leftmargin=.5in}

\newsiamremark{remark}{Remark}
\newsiamremark{hypothesis}{Hypothesis}
\crefname{hypothesis}{Hypothesis}{Hypotheses}
\newsiamthm{claim}{Claim}

\headers{}{Hao Wang, Hao Zeng, Jiashan Wang}

\title{Relating $\ell_p$ regularization and reweighted $\ell_1$ regularization \thanks{Submitted to the editors DATE.
}}

\author{Hao Wang\thanks{School of Information Science and Technology, ShanghaiTech
University, Shanghai, China
  (\email{haw309@gmail.com}).}
\and Hao Zeng\thanks{School of Information Science and Technology, ShanghaiTech
University, Shanghai, China
  (\email{zenghao@shanghaitech.edu.cn}).}
\and Jiashan Wang\thanks{Department of Mathematics, University of Washington, Seattle, USA}  (\email{jsw1119@gmail.com}).}

\usepackage{amsopn}

\ifpdf
\hypersetup{
  pdftitle={Relating $\ell_p$ regularization and reweighted $\ell_1$ regularization },
  pdfauthor={Hao Wang, Hao Zeng, Jiashan Wang}
}
\fi

\begin{document}

\maketitle

\begin{abstract} We propose a general framework of iteratively reweighted $\ell_1$ methods for solving $\ell_p$ regularization problems. 
We prove that after some iteration $k$, the iterates generated by the proposed methods 
have the same support and sign as the limit points, and are bounded away from 0, so that 
the algorithm behaves like solving a smooth problem in the reduced space. 
As a result, the global convergence can be easily obtained and an update 
strategy for the smoothing parameter is proposed which can  automatically terminate
the updates for zero components. 
  We show that $\ell_p$ regularization problems 
  are locally equivalent to a weighted $\ell_1$ regularization problem and every optimal 
  point corresponds to a   Maximum A Posterior  estimation for independently and non-identically distributed 
  Laplace prior  parameters.  Numerical experiments exhibit the behaviors and the efficiency of our proposed methods.   
 \end{abstract}



\section{Introduction}
In recent years, sparse regularization problems have attracted considerable attentions because of their wide applications, including machine learning \cite{liu2007sparse,scardapane2017group}, statistics \cite{hastie2015statistical,li2018feature} and compressed sensing \cite{candes2008enhancing, yin2015minimization}.   
Sparse solutions generally lead to better generalization of the model performance from training data to future data. 
A common approach is the $\ell_p$ ($0\le p \le 1)$ regularization technique, which
minimizes  the loss function combined with a convex/nonconvex penalization term such as the $\ell_p$  norm of the model parameters.   
Nonconvex regularization technique with $0<p<1$ nowadays has become popular due to its power in promoting sparsity.

The primary focus of this paper is on analyzing the properties of nonconvex $\ell_p$ regularization, and designing efficient numerical 
algorithms for solving the   $\ell_p$ regularized problem 
\begin{equation}\label{prob.lpproblem}\tag{P}
\begin{aligned}
& \underset{x\in \mathbb{R}^n}{\min} 
& & F(x):= f(x)+ \lambda \psi(x)   \quad \text{with }  \psi(x)=\|x\|_p:= \left(\sum_{i=1}^n |x_i|_p\right)^{1/p}
\end{aligned}
\end{equation}
where $f:\mathbb{R}^n\to \mathbb{R}$ is a continuously differentiable function, $p\in(0,1)$ 
  and $\lambda > 0$ is the regularization parameter. 
This technique is often regarded as a better approximation to the $\ell_0$ regularization than the $\ell_1$ regularization, and   can often 
yield sparser solutions. 


  However, $\ell_p$ regularized problems are  generally difficult to handle and analyze  due to its nonconvex and 
non-Lipschitz continuous  nature.   In fact, Ge \cite{ge2011note} proved that finding the global minimal value of the problem with $\ell_p$-norm $(0<p<1)$ regularization term is strongly NP-Hard. 
Therefore, many works focus on replacing  the nonconvex and nonsmooth regularization 
term  with trackable smooth approximation. For example, 
Chen et al. \cite{chen2010lower} approximate $|x_i|^p$ by   a continuously differentiable function 
\begin{equation*}\label{appro-term1}
\psi_{\mu}(x_i) = 
\begin{cases}
|x_i|, & |x_i|>\mu, \\ 
\dfrac{x_i^2}{2 \mu} + \dfrac{\mu}{2}, & |x_i| \leq \mu,
\end{cases}
\end{equation*}
with $\mu \in \mathbb{R}_+$, which is solved by a  hybrid orthogonal matching pursuit-smoothing gradient method. 
 Lu \cite{lu2014iterative} constructed another  Lipschitz continuous approximation
to $|x_i|^p$
\begin{equation}\label{appro-term2}
\psi_{\mu_{\epsilon}}(x_i) = 
\begin{cases}
|x_i|^p, & |x_i|^p > \mu_{\epsilon}; \vspace{0.1cm}\\
\mu_{\epsilon}, & |x_i|^p \leq \mu_{\epsilon}.
\end{cases}
\end{equation}
with $ \mu_{\epsilon} = \dfrac{\epsilon}{\lambda n}$ and then proposed an iteratively reweighed algorithm.
Chen  \cite{Chen13} proposed a smoothing trust region Newton algorithm for solving the approximated 
problem by replacing $|x_i|^p$ by 
\[
\psi_{\mu}(x_i) = \left(x_i^2 + 4 \mu^2 \right)^{\frac{p}{2}}. 
\]
Another type of approximation technique is to add smoothing perturbation to each $|x_i|$, which mainly includes 
the $\epsilon$-approximation of \eqref{prob.lpproblem} 
\begin{equation}\label{l2-lp-appro1}
 \psi_{\epsilon,1}(x_i)  =  \left( |x_i| + \epsilon \right)^p  \  \text{ and } \    \psi_{\epsilon,2}(x_i) = \left( x_i^2 + \epsilon \right)^{p/2}
\end{equation}
by Chen and Zhou
 \cite{chen2010convergence} 
 and   Lai and Wang \cite{lai2011unconstrained}  
with prescribed small $\epsilon > 0$. 

%

Among these algorithms, iteratively reweighted methods \cite{lu2014iterative,sun2017global,wang2018nonconvex}  for solving 
approximation \eqref{l2-lp-appro1} are popular and proved to be efficient for many cases. 
  At each iteration, it replaces  \eqref{l2-lp-appro1}   by 
\begin{equation}\label{first.l2}
 p(|x_i^k|+\epsilon_i)^{p-1}|x_i|\quad\text{or}\quad \tfrac{p}{2}((x^k_i)^2+\epsilon_i)^{\frac{p}{2}-1}x_i^2,
 \end{equation}
 respectively  
via linearizing   $(\cdot)^p$ and $(\cdot)^{\frac{p}{2}}$ at $x^{k}$. 
 In this problem, large $\epsilon$ will smooth out many local minimizers, while small values make the subproblems difficult to solve and easily trapped into bad local minimizers. In order to approximate \cref{prob.lpproblem} effectively, Lu \cite{lu2014iterative} improved these weights by dynamically updating perturbation parameter $\epsilon_i$ at each iteration to better approximate original problem.

\subsection{Key contributions}
The contributions of this paper can be summarized below. 
\begin{itemize}
	\item We proposed a general framework of iteratively $\ell_1$ methods and studied the convergence, which can include different types of iteratively reweighted $\ell_1$ methods such as first-order and second-order methods. 
	\item  We showed that the proposed  iteratively reweighed $\ell_1$ methods locally  have the same support and sign of the iterates as the optimal solution when applied to non-Lipschitz regularization problems. Consequently, these methods locally behave like solving 
	a smooth problem, which could potentially make  the analysis for these algorithms easier and straightforward.  
		\item We showed that the $\ell_p$ regularization problem is locally equivalent to a weighted $\ell_1$ regularization problem.  That being said,
  any first-order optimal solution of $\ell_p$ regularization problem can be identified with the optimal solution of a weighted $\ell_1$ regularization problem which is equivalent to finding 
  a mode of  Maximum A Posterior (MAP) for independently and non-identically distributed 
  Laplace prior on the parameters.


\end{itemize}

\subsection{Notation}
For $x\in \mathbb{R}^n$, let $x_i$ be the $i$th element of $x$, and 
define the support of $x$ as 
$\Ical(x) = \{ i \mid x_i \neq 0\}$ and its complement as 
$\Acal(x) = \{ i \mid x_i = 0\}$. 
Denote $e$ as the vector of all 1s of appropriate dimension. 
The sign  of $x\in\mathbb{R}^n$  is defined as $(\text{sign}(x))_i = \text{sign}(x_i)$. 
For $H\in\mathbb{R}^{n\times n}$ and index sets $\Acal, \Ical \subset \{1,\ldots, n\}$, 
let $H_{\Acal, \Ical}$ be the matrix consisting of $h_{i,j}, i\in\Acal, j\in\Ical$, 
and $\text{diag}(a_i, i\in \Acal)$ be the diagonal matrix with the elements of vector 
$a_i, i\in\Acal$  on the main diagonal.
The componentwise  product of two vectors $a\in\mathbb{R}^n$ and $b\in\mathbb{R}^n$ is defined as 
$(a\circ b)_i = a_i b_i$.  
Let  $\{-1,0,+1\}^n$ be the set of $n$-dimensional vectors with components
chosen from $\{-1, 0, +1\}$. 

In $\mathbb{R}^n$, denote $\|\cdot\|_p$ as the $\ell_p$ norm with $p\in(0,+\infty)$, i.e., 
$\|x\|_p = \left(\sum_{i=1}^n |x_i|^p\right)^{1/p}$.   Note  that for $p\in(0,1)$, 
this  does not define a proper norm due to its lack of subadditivity. 
If function $f: \mathbb{R}^n \to \bar{ \mathbb{R}}:=\mathbb{R} \cup \{+\infty\}$
is convex, then the subdiferential of $f$ at $\bar x$ is given by 
\[ \partial f(\bar x):=\{ z \mid f(\bar x) + \langle z, x-\bar x\rangle \le f(x),  \  \forall x\in\mathbb{R}^n\}.\]
In particular, for $x\in \mathbb{R}^n$, we use $\partial \|x\|_1$ to denote the set 
$\{ \xi \in \mathbb{R}^n \mid  \xi_i \in \partial |x_i|, i=1,\ldots, n\}.$
For closed convex set $\Omega\subset \mathbb{R}^n$, define the Euclidean distance of point $a\in\mathbb{R}^n$
to $\Omega$ as 
$\text{dist}(a,\Omega) = \min_{b\in \Omega} \| a - b\|_2$.

\section{Iteratively reweighted $\ell_1$ methods} \label{sec.l1lp}
In this section, we introduce the framework of iteratively reweighted $\ell_1$ methods for  solving  \eqref{prob.lpproblem}. 
Given $\epsilon\in\mathbb{R}^n_{++}$, the iteratively reweighted $\ell_1$ method is based on smooth approximation $F(x, \epsilon)$ of $F(x)$
\[ F(x, \epsilon) : = f(x) + \lambda \sum_{i=1}^n (|x_i|+\epsilon_i)^p.\]

We   make the following assumption about $f$. 
\begin{assumption}\label{ass.lip}
	$f$ is Lipschitz differentiable with constant $L_f \ge 0$.
\end{assumption}

At  $k$-th iteration, the algorithm 
formulates a convex local model to approximate $F(x)$
\[ G(x; x^k, \epsilon^k) : = Q_k(x) + \lambda \sum_{i=1}^n w(x_i^k, \epsilon_i^k) |x_i| \] 
where the weights are given by 
$w(x_i^k, \epsilon_i^k) = p(|x_i^k|+\epsilon_i^k)^{p-1}.$ $Q_k(x)$ represents a local approximation model to $f$ at $x^k$, and is generally assumed to be smooth and convex. 
Common approaches include the following.
\begin{itemize}
	\item Proximal first-order approximation: 
	$Q_k(x) = \nabla f(x^k)^T(x-x^k) +  \frac{\beta}{2}\|x-x^k\|_2^2$ with $\beta> 0$. 
	\item  Quasi-Newton approximation: 
	$Q_k(x) = \nabla f(x^k)^T(x- x^k) + \tfrac{1}{2}(x-x^k)^TB^k(x-x^k)$ with $B^k\in\mathbb{R}^{n\times n}$. 
	\item Newton approximation: 
	$Q_k(x) = \nabla f(x^k)^T(x- x^k) + \tfrac{1}{2}(x-x^k)^T \nabla^2f(x^k)(x-x^k)$.  
\end{itemize}
The next iterate $x^{k+1}$ is then computed as the solution of  $\min_{x\in\mathbb{R}^n}  G(x; x^k,\epsilon^k)$: 
\[ x^{k+1} \gets \arg\min_{x\in\mathbb{R}^n} G(x; x^k, \epsilon^k)\]
with $\epsilon$  driven towards to 0:  $\epsilon^{k+1} \le  \mu \epsilon^k$ and $\mu\in(0,1)$. 

We state the framework of this iteratively reweighted $\ell_1$ method in  Algorithm  \ref{alg.framework}.

\begin{algorithm}[htp]
	\caption{General  framework of iteratively reweighted $\ell_1$ (IRL1) methods}
	\label{alg.framework}
	\begin{algorithmic}[1]
		\STATE{\textbf{Input:} $\mu\in(0,1), \epsilon^0\in\mathbb{R}^n_{++}$ and $x^0 $.}
		\STATE{\textbf{Initialize: set $k=0$}. }
		\REPEAT
		\STATE Reweighing:  $w(x_i^k, \epsilon_i^k)= p(|x^k_i|+\epsilon^k_i)^{p-1}. $
		\STATE Compute new iterate: $x^{k+1} \gets \underset{x\in \mathbb{R}^n}{\text{argmin}} \quad \bigg\{Q_k(x) + \lambda \sum\limits_{i=1}^{n} w(x_i^k, \epsilon_i^k) |x_i| \bigg\}.$ 
		\STATE Set $\epsilon^k \le \mu \epsilon^{k-1}$, $k\gets k+1$. 
		
		\UNTIL{convergence}
	\end{algorithmic} 
\end{algorithm}

We make the following assumptions about the choice of $Q_k(\cdot)$. 
\begin{assumption} \label{ass.basic} 
	The initial point $(x^0, \epsilon^0)$ and local convex model  $Q_k(\cdot)$ are  such that    
	\begin{enumerate}
		\item[(i)] 
		The level set 
		$ \Lcal(F^0):= \{ x \mid F(x)  \le F^0\}$ is bounded where $F^0 := F(x^0, \epsilon^0)$.
		\item[(ii)] For all $k\in\mathbb{N}$,  $\nabla Q_k(x^k) = \nabla f(x^k)$, 
		$Q_k(\cdot)$ is strongly convex with constant $M>L_f/2 > 0$, and 
		Lipschitz differentiable with constant $L> 0$. 
	\end{enumerate}
\end{assumption}

This assumption is relatively loose  on the local model  $Q_k(\cdot)$.  In particular, in the proximal method, this condition trivially holds.  In the (quasi-)Newton approximation, it suffices to require  $MI \preceq \nabla^2 f(x^k) \preceq L I$ ($MI \preceq B^k \preceq L I$). 
It should be noticed that  in fact our analysis only relies on these conditions to hold on $ \Lcal(F^0)$.

\subsection{Monotonicity of $F(x, \epsilon)$} In this section, we show that $F(x, \epsilon)$ is monotonically decreasing over our iterates $(x^k, \epsilon^k)$. For the ease of presentation, we define the following two terms
\[
\begin{aligned}
\Delta F(x^{k+1}, \epsilon^{k+1}): = & F(x^k, \epsilon^k)-F(x^{k+1}, \epsilon^{k+1})\\
\Delta G(x^{k+1}; x^k, \epsilon^k): = & G(x^k; x^k, \epsilon^k) - G(x^{k+1}; x^k, \epsilon^k),
\end{aligned}
\]
and use the following  shorthands   
$ w_i^k: = w(x_i^k, \epsilon_i^k) \text{ and } W^k:= \text{diag}(w_1^k, \ldots, w_n^k).$



\begin{proposition}\label{prop.sign} Suppose  \cref{ass.lip} and \ref{ass.basic} hold. Let $\{(x^k, \epsilon^k)\}$ be the sequence generated by \cref{alg.framework}.  It follows that $F(x, \epsilon)$ is monotonically decreasing over $\{(x^k,\epsilon^k)\}$ and the reduction satisfies
	\begin{align}	
	F(x^0, \epsilon^0) - F(x^k; \epsilon^k) & \ge (M-\tfrac{L_f}{2}) \sum_{t=0}^{k-1} \|x^{t+1}-x^t\|^2_2. 
	\label{eq:convergence3}
	\end{align}
	Moreover,	$\lim\limits_{k\to\infty}\|x^{k+1} - x^k\|_2 = 0$, and there exists $C> 0$ such that 
	$\|\nabla Q_k(x^{k+1})\|_\infty \le C$ for any $k \in \mathbb{N}$.
	
\end{proposition}
\begin{proof}
	From \cref{ass.lip}  and 
	\ref{ass.basic}, we have
	\[ 
	\begin{aligned} 
	f(x^k) - f(x^{k+1}) \ge & \ \nabla f(x^k)^T(x^k-x^{k+1}) - \frac{L_f}{2} \|x^k - x^{k+1} \|_2^2 \\
	Q_k(x^k) - Q_k(x^{k+1} )  \le  & \  \nabla f(x^k)^T  (x^k -  x^{k+1} ) - \frac{M}{2} \|x^{k+1} - x^k\|_2^2.
	\end{aligned}
	\]
	
	It follows that 
	\begin{equation}\label{q les f} 
	f(x^k) - f(x^{k+1}) \ge  Q_k(x^k) - Q_k(x^{k+1}) + \frac{M- L_f}{2} \|x^k-x^{k+1}\|_2^2.
	\end{equation}
	
	On the other hand, the concavity of $a^p$ on $\mathbb{R}_{++}$ gives 
	$a_1^p \le a_2^p + pa_2^{p-1}(a_1- a_2)$ for any $a_1, a_2 \in \mathbb{R}_{++}$, implying for  $i= 1,\ldots, n$
	\[\begin{aligned}
	(|x_i^{k+1}|+\epsilon_i^k)^p \le &\ (|x_i^k |+\epsilon_i^k)^p + p (|x_i^k |+\epsilon_i^k)^{p-1} (|x_i^{k+1} | - |x_i^k|)\\
	= &\ (|x_i^k |+\epsilon_i^k)^p +  w_i^k (|x_i^{k+1} | - |x_i^k|).
	\end{aligned} \] 
	Summing the above inequality over   $i$ yields 
	\begin{equation}\label{w les w} 
	\sum_{i=1}^n ( |x_i^{k+1}|+\epsilon_i^k)^p   \le  \sum_{i=1}^n (|x_i^k |+\epsilon_i^k)^p +  \sum_{i=1}^n w_i^k (|x_i^{k+1} | - |x_i^k|). 
	\end{equation}
	Combining \eqref{q les f} and \eqref{w les w} gives 
	\begin{equation}\label{eq:convergence1}
\begin{aligned}
\Delta F(x^{k+1}, \epsilon^{k+1}) \geq  \Delta G(x^{k+1}; x^k, \epsilon^k) + \frac{M- L_f}{2} \|x^k-x^{k+1}\|_2^2.
\end{aligned}
	\end{equation}
	 \cref{ass.basic} implies the subproblem solution $x^{k+1}$  satisfies the  optimality condition  
	\begin{equation}\label{kkt subproblem}
	\nabla   Q_k(x^{k+1}) + \lambda  W^k \xi^{k+1}  = 0
	\end{equation}
	with $\xi^{k+1} \in \partial\|x^{k+1}\|_1$. Hence 
    \begin{equation} \label{eq:convergence2}
    \begin{aligned}
    & \ G(x^k; x^k, \epsilon^k) - G(x^{k+1}; x^k, \epsilon^k)\\
    = & \  Q_k(x^k) - Q_k(x^{k+1}) + \lambda \sum_{i=1}^n 
    w_i^k ( |x^k_i| -   |x_i^{k+1}| ) \\
    \ge &\  \nabla Q_k(x^{k+1})^T(x^k-x^{k+1}) +  \tfrac{M}{2} \|x^{k+1} - x^k\|_2^2 + \lambda \sum_{i=1}^n 
    w_i^k\xi_i^{k+1} (x^k_i -   x_i^{k+1})\\
    = &\  [\nabla Q_k(x^{k+1}) + \lambda W^k\xi^{k+1}]^T(  x^k - x^{k+1}) + \tfrac{M}{2} \|x^{k+1} - x^k\|_2^2\\
    = & \ \tfrac{M}{2} \|x^{k+1} - x^k\|_2^2,
    \end{aligned} 
    \end{equation}   
where the   inequality is by  \cref{ass.basic} and the convexity of $|\cdot|$, and  the last equality  is by \eqref{kkt subproblem}.

	 We then combine \eqref{eq:convergence1} and \eqref{eq:convergence2} to get 
	\begin{equation}\label{eq:convergence4} 
	\Delta F(x^{k+1}, \epsilon^{k+1})  \ge   F(x^{k}, \epsilon^{k}) - F(x^{k+1}, \epsilon^{k})  \ge   (M-   \tfrac{ L_f}{2} )\|x^k - x^{k+1}\|_2^2.
	\end{equation}
	Replacing $k$ with $t$ and summing up from $t=0$ to $k-1$, we have 
	\[\sum_{t=0}^{k-1}(  F(x^{t}, \epsilon^{t})  -  F(x^{t+1}, \epsilon^{t+1}) ) \ge  (M-   \tfrac{ L_f}{2} )  \sum_{k=0}^{k-1}  \|x^t -x^{t+1}\|_2^2,\]
	completing the proof of \eqref{eq:convergence3}.

	It follows that 
	$\{x^k\}\subset \Lcal(F^0)$ by $F(x^k) \le F(x^k, \epsilon^k) \le F(x^0, \epsilon^0)$ by \eqref{eq:convergence3}. 
	By  \cref{ass.basic}(i),  it is bounded and $\liminf\limits_{k\to\infty}F(x^k, \epsilon^k) > - \infty$. 
	Taking $k\to \infty$ in \cref{eq:convergence1}, we know 
	\[ (M-\tfrac{L_f}{2}) \sum_{i=0}^{\infty} \|x^{k+1}-x^k\|^2 \le  F(x^0, \epsilon^0) - \liminf\limits_{k\to\infty}F(x^k, \epsilon^k) < \infty,\]
	implying  $\lim_{k\to \infty} \|x^{k+1}-x^{k}\| = 0$. 
	Moreover, it follows from the boundedness of $\{x^k\}$ that 
	there must exist $C>0$ such that  $\|Q_k(x^{k+1})\|_\infty \le C$ for any $k\in\mathbb{N}$. 
\end{proof}

%

\subsection{Locally stable sign}\label{sec.local sign}


We now show that under Assumption~\ref{ass.basic}, after some iteration, the support of the iterates remains unchanged.  The result is summarized in the following theorem. 

\begin{theorem}[Locally stable support]\label{thm.stable.support} Assume Assumption \ref{ass.lip} and \ref{ass.basic} hold and let $\{(x^k, \epsilon^k)\}$ be a sequence generated by \cref{alg.framework}.   $C$ is the constant as defined in \cref{prop.sign}.  Then we have the following
\begin{enumerate}
		\item[(i)]   If $w(x_i^{\tilde k}, \epsilon_i^{\tilde k}) > C/\lambda$ for some ${\tilde k}\in \mathbb{N}$,  then  $x_i^k \equiv 0$ for all $k > \tilde k$. Conversely, 
		if there exists $\hat k>\tilde k$ for any $\tilde k \in\mathbb{N}$ such that $x_i^{\hat k} \neq 0$,  then $w(x_i^k, \epsilon_i^k) \le  C/\lambda$ for all $k\in\mathbb{N}$. 
		\item[(ii)]       There  exist index sets $\Ical^*\cup \Acal^* = \{1,\ldots,n\}$ and $\bar k > 0$, such that  $\forall \  k> \bar k$,  
		$\Ical(x^k)\equiv \Ical^*$ and $\Acal(x^k) \equiv \Acal^*$.
		\item[(iii)] 	For any $i\in\Ical^*$,  it holds that 
		\begin{equation}\label{eq.xboundeps} |x_i^k| > \left(\frac{C}{p \lambda }\right)^{\frac{1}{p-1}} - \epsilon_i^k > 0,\quad   i\in\Ical^*. \end{equation} 
		Therefore, $\{|x_i^k|, i\in\Ical^*, k\in\mathbb{N}\}$ are bounded away from 0 after some $\hat k \in \mathbb{N}$. 
		\item[(iv)]  For any cluster point $x^*$ of $\{x^k\}$, it holds that $\Ical(x^*) = \Ical^*$, $\Acal(x^*) = \Acal^*$ and  
		\begin{equation}\label{eq.xbound} |x^*_i| \ge   \left(\frac{C}{p \lambda }\right)^{\frac{1}{p-1}},\quad i\in \Ical^*. \end{equation}
	\end{enumerate}
\end{theorem}
\begin{proof}

	(i) If $w(x_i^{\tilde  k}, \epsilon_i^{\tilde  k}) > C/\lambda $ for some ${\tilde  k} \in \mathbb{N}$,  then  the optimality condition \eqref{kkt subproblem} 
	implies $x_i^{{\tilde  k} +1} = 0$. Otherwise we have 
	$ | \nabla_i Q_{\tilde  k}(x^{\tilde  k +1}) | =   \lambda w(x_i^{\tilde  k}, \epsilon_i^{\tilde  k})  > C $, 
	contradicting \cref{prop.sign}.  Monotonicity of $(\cdot)^{p-1}$ and  $0 + \epsilon_i^{\tilde k +1} \le |x_i^{\tilde k}| + \epsilon_i^{\tilde k}$ yield  
	\[w(x_i^{{\tilde  k}+1}, \epsilon_i^{{\tilde  k}+1}) =  p(0+ \epsilon_i^{{\tilde  k}+1})^{p-1}  \ge   p( |x_i^{\tilde k}| + \epsilon_i^{\tilde  k})^{p-1}  =   w(x_i^{\tilde  k}, \epsilon_i^{\tilde  k}) > C.\]
	By induction  we know that  
	$x_i^k \equiv 0$ for any $k > \tilde  k$.  This completes the proof of (i). 
	
	
	%
	%
	
	(ii)  Suppose by contradiction this statement is not true.  There exists $j \in \{1,\ldots, n\}$ such 
	that $\{x_j^k\}$ takes zero and nonzero values both for infinite times.  
	Hence, 
	there exists a subsequence $\Scal_1 \cup \Scal_2 = \mathbb{N}$ such that 
	$|\Scal_1|=\infty$, $|\Scal_2|=\infty$  and that  
	\[ x_j^k = 0, \forall k \in \Scal_1 \ \text{ and }\  x_j^k \neq 0, \forall k \in  \Scal_2.\]
	Since $\{ \epsilon_j^k \}_{\Scal_1}$ is monotonically decreasing to 0, there exists $\tilde k \in \Scal_1$ such that 
	\[ w(x_j^{\tilde k} , \epsilon_j^{\tilde k}) = p(|x_i^{\tilde k}| + \epsilon_j^{\tilde k} )^{p-1} = p (\epsilon_j^{\tilde k})^{p-1} > C/\lambda. \]
	It follows that $x_j^k \equiv 0$ for any $k > \tilde k$ by (i) which implies $\{ \tilde k + 1, \tilde k + 2, \ldots\} \subset \Scal_1$ and $|\Scal_2| < \infty$. This violates the assumption   $|\Scal_2| = \infty.$ Hence, (ii) is true.

	(iii)  Combining  (i) and (ii), we know for any   $i\in \Ical^*$, $w_i^k \le  C/\lambda$, which is equivalent to \eqref{eq.xboundeps}. This proves (iii). 

	(iv) For $i\in \Acal^*$,  (ii) implies that $i\in \Acal(x^*)$. 
	For $i \in \Ical^*$, (ii) and (iii) imply that   \eqref{eq.xbound}  is true, meaning 
	$i \in \Ical(x^*)$. 
\end{proof}

The above theorem indicates an interesting property of the iterates generated by \cref{alg.framework}.   All the cluster points of the iterates 
have the same support, so that we can use 
$\Ical^*= \Ical(x^*), \Acal^*= \Acal(x^*)$. 
%
We continue to  show the signs of  $\{x^k\}$ also  remain unchanged for sufficiently large $k$.    
Combined with \cref{thm.stable.support}, 
this means the iterates $\{x^k_{\Ical^*}\}$ will eventually 
stay in the interior of the same orthant.  This result is shown in the following theorem. 

\begin{theorem}[Locally stable sign]\label{thm.stable.sign}  
	Suppose $\{x^k\}$ be a sequence generated by \cref{alg.framework} and Assumptions  \ref{ass.lip} and \ref{ass.basic} are satisfied.  
	There exists $\bar k  \in \mathbb{N}$, such that 
	the sign of $\{x^k\}$ are fixed for any $k > \bar k$, i.e.,  
	$\text{sign}(x^k)\equiv s$ for some  $s \in \{-1,0,+1\}^n$.
\end{theorem}
\begin{proof} From \cref{thm.stable.support}, we only have to show that the sign of 
	$x_i^k, i\in\Ical^*$ is fixed for sufficiently large $k$.  
	By \cref{prop.sign} and   \cref{thm.stable.support}(iii), there exists   $\bar k \in\mathbb{N}$, such that  	for any $k> \bar k$
	\begin{align}
	&\ \|x^{k+1} - x^k\|_2 <  \ \bar \epsilon: =  \frac{1}{2} \left(\frac{C}{p \lambda }\right)^{\frac{1}{p-1}} \label{contradiction here} \\ 
	\text{ and } & \  |x_i^k| >   \  \bar \epsilon, \quad  \forall i\in \Ical^*. \label{coro.support}
	\end{align}
  
      We prove this by contradiction. Assume there exists   $j\in\Ical^*$ such that  the sign of $x_j$ 
	changes after $\bar k$.  
	Hence there must be  $\hat k  \ge   \bar k  $ such that $x_j^{\hat k}x_j^{\hat k+1} < 0 $. 
         It follows that  
	\[  \|x^{\hat k+1} - x^{\hat k}\|_2  \ge  |x_j^{\hat k+1} - x_j^{\hat k}| = \sqrt{ (x_j^{\hat k+1} )^2   + (x_j^{\hat k})^2  - 2x_j^{\hat k} x_j^{\hat k+1} }
	> \sqrt{ \bar\epsilon ^2   + \bar\epsilon ^2}
	=    \sqrt{2}\bar\epsilon,\]
	where the last inequality is by \cref{coro.support}.  
	This contradicts  with \eqref{contradiction here}, completing the proof.  
\end{proof}

The locally stable support and sign of the iterates imply that 
for sufficiently large $k$,    the algorithm is equivalent to solving a smooth problem in the reduced space $\mathbb{R}^{\Ical^*}$. 
Our analysis in the remainder of this paper is based on this observation.

\subsection{Global convergence}\label{sec.global}





We now provide the convergence of the framework of  iteratively weighted $\ell_1$ method. 

The first-order necessary condition \cite{lu2014iterative} of \eqref{prob.lpproblem} is 
\begin{equation} \label{eq:optimalcondition}
  \nabla_if(x^*)+\lambda p |x_i^*|^{p-1} \sign( x_i^*) =0 \quad \text{for} \quad i\in\Ical(x^*).
\end{equation}
The following theorem shows that every limit point of 
the iterates is a first-order stationary solution.

\begin{theorem}\label{th:limitpoint}  Suppose 
	 \cref{ass.lip} and 
	\ref{ass.basic} hold.  Let $\{x^k\}$ be a sequence generated by \cref{alg.framework} and $\Omega$ be the set 
	of limit points of $\{x^k\}$. 
	Then  $\Omega \neq \emptyset$
	and any $x^* \in \Omega$ is first-order optimal for \cref{prob.lpproblem}.
	Moreover,  any $x^*\in\Omega$ with $\Acal(x^*)\neq \emptyset$ is not a maximizer of \cref{prob.lpproblem}.
\end{theorem}
\begin{proof}  
	Boundedness of $\Lcal(F^0)$ from \cref{ass.basic} implies $\Omega \neq \emptyset$.
	Let $x^*$ be a limit point of $\{x^k\}$ with subsequence $\{x^k\}_{\Scal} \to x^*$. 	
	From Theorem \ref{thm.stable.support} and  \ref{thm.stable.sign}, there exists $\bar k\in\mathbb{N}$ such that 
	for any   $k > \bar k$, the sign of $x^k$ stays the same. 
	
	Optimality condition of subproblem yields
	\[
	\nabla_iQ_k(x^{k+1})+ \lambda p(|x_i^k|+\epsilon_i^k)^{p-1}\text{sign}(x_i^{k+1})  = 0,  \quad    i\in \Ical(x^*). 
	\]
	Taking the limit on $\Scal$, we have 
	for each $   i\in \Ical(x^*)$, 
	\[\begin{aligned}
	0  = &  \lim_{k\to\infty \atop k\in\Scal} |  \nabla_iQ_k(x^{k+1})+ \lambda p(|x_i^k|+\epsilon_i^k)^{p-1} \text{sign}(x_i^{k+1}) |  \\
	\ge &  \lim_{k\to\infty \atop k\in\Scal} | \nabla_iQ_k(x^{k}) + \lambda p(|x_i^k|+\epsilon_i^k)^{p-1} \text{sign}(x_i^{k}) |  - | \nabla_iQ_k(x^{k+1}) - \nabla_iQ_k(x^{k})|  \\
		\ge &  \lim_{k\to\infty \atop k\in\Scal} | \nabla_iQ_k(x^{k}) + \lambda p(|x_i^k|+\epsilon_i^k)^{p-1} \text{sign}(x_i^{k}) |  - \| \nabla Q_k(x^{k+1}) - \nabla Q_k(x^{k})\|_1  \\
				\ge &  \lim_{k\to\infty \atop k\in\Scal} | \nabla_iQ_k(x^{k}) + \lambda p(|x_i^k|+\epsilon_i^k)^{p-1} \text{sign}(x_i^{k}) |  - \sqrt{n}\| \nabla Q_k(x^{k+1}) - \nabla Q_k(x^{k})\|_2  \\
	\ge &     \lim_{k\to\infty \atop k\in\Scal}	|  \nabla_if(x^k) +\lambda p(|x_i^k|+\epsilon_i^k)^{p-1} \text{sign}(x_i^{k}) |  - \sqrt{n} L \| x^{k+1} - x^{k}\|_2 \\
	=   &\ \nabla_if(x^*) + \lambda p |x_i^*|^{p-1} \sign(x_i^*),
	\end{aligned} \]
	where the second inequality is due to 
	\[   \| \nabla Q_k(x^{k+1}) -   \nabla Q_k(x^k)\|_2 \le   L \|x^{k+1} - x^k\|_2,\]
	by  \cref{prop.sign} and Assumption \ref{ass.basic}.	Therefore, $x^*$ is first-order optimal. 
\end{proof}

%

%


\subsection{Uniqueness of limit points}\label{sec.unique}

By Theorem \ref{thm.stable.support} and  \ref{thm.stable.sign},   $x^k_{\Ical^*}\equiv 0$ for sufficiently, meaning 
 the IRL1 algorithm behaves like solving a smooth problem on 
the reduced space $\mathbb{R}^{\Ical^*}$ . We can thus derive various conditions that guarantee  
the  uniqueness of the limit points.  For example, we can  show  the uniqueness of limit points 
under \emph{Kurdyka-\L ojasiewicz} (KL) property  \cite{attouch2013convergence,bolte2014proximal} of $F$, which is generally believed to be a weak assumption 
needed in the analysis for many algorithms. However, due to limit of space, we only  provide the following 
sufficient condition to guarantee the uniqueness of a limit point $x^*$ of $\{x^k\}$.

\begin{theorem}
	Suppose \cref{ass.lip} and 
	\ref{ass.basic} are true.  Then  $\{x^k\}$ either has a unique limit point, or its limit points form
	 a compact connected set contained in the same orthant on which the objective has the same value.  
	 In particular,  at a limit point $x^*$, if 
	$[\nabla^2 f(x^*)]_{\Ical^*, \Ical^*} + \lambda diag( \sign(x_i^*)p(p-1)(x_i^*)^{p-2}, i\in \Ical^*)$ is 
	nonsingular, 
	then $x^*$ is the unique limit point. 
 \end{theorem}
\begin{proof}
 If there exist multiple cluster points  for $F( [x_{\Ical^*}; 0_{\Acal^*}])$,   
	we have from  $\|x^{k+1}-x^k\|_2 \to 0$ by \cref{prop.sign} and  \cite[Lemma 2.6]{fejer} 
	that the set of cluster points of $\{x^k\}$ is a compact connected set.
	
	On the other hand, it is obvious that for any $x^*\in\Omega$ satisfying \eqref{eq:optimalcondition},
	 $x^*_{\Ical^*}$  is the optimal solution of 
the reduced problem of $\ell_p$ regularization 
\[ \min_{x_{\Ical^*}\in \mathbb{R}^{\Ical^*}} \ F( [x_{\Ical^*}; 0_{\Acal^*}]):= f([x_{\Ical^*}; 0_{\Acal^*}])+\lambda\sum_{i\in \Ical^*}  |x_i|^p.\]

	By \cite[Theorem 7.3.5]{patrikalakis2009shape}, if 
	\[ [\nabla^2 F( [x_{\Ical^*}; 0_{\Acal^*}])]_{\Ical^*, \Ical^*}=
	[\nabla^2 f(x^*)]_{\Ical^*, \Ical^*} + \lambda diag( \sign(x_i^*)p(p-1)(x_i^*)^{p-2}, i\in \Ical^*) \]
	is 
	nonsingular at $x_{\Ical^*}^*$,  then $x^*$ is an isolated critical point.
	 However, we have shown  that $\Omega^*$ is a compact connected set and each element is a critical point---a contradiction.  
	 Therefore, $x^*$ must be the unique limit point.  
\end{proof}

\subsection{Smart  $\epsilon$ updating strategies}
\label{sec.updating}

For iteratively reweighted methods, it is helpful to start with 
a relatively large $\epsilon$ and gradually reduce it to 0, since this may prevent the algorithm 
from quickly getting trapped into a local minimum.  
However, as the iteration proceeds,  we need to let $\epsilon_i \to 0, i\in \Ical(x^*)$ 
  to obtain convergence, and keep $\epsilon_i \to 0, i\in \Acal(x^*)$ updated slowly or even 
fixed after some iterations to prevent  potential numerical issues or from becoming stuck at a local minimum. 
Such a strategy may need the estimate of $\Ical(x^*)$ and $\Acal(x^*)$, which are generally unknown 
at the beginning.
The updating strategy, named as ``smart reweighting'',   is  as follows. 
\begin{equation}
\boxed{ 
\tag{SR}\label{update.eps}
\begin{cases}
\epsilon_i^{k+1} = \ \ \epsilon_i^k  &\quad \text{ if  } \  x_i^{k+1} = 0, \\
\epsilon_i^{k+1} \le \mu \epsilon_i^k &\quad \text{ if  } \   x_i^{k+1} \ne 0. 
\end{cases}
}
 \end{equation}

If  we update  $\epsilon$  in  \cref{alg.framework} according to \eqref{update.eps}, one can easily see that 
\cref{prop.sign} still holds true. Furthermore, we have the following results. 

\begin{theorem}
	Suppose \cref{ass.lip} and 
	\ref{ass.basic} are true, and  $\{x^k\}$ are generated by \cref{alg.framework} with $\epsilon$ updated according 
	to \eqref{update.eps}.  
	The following hold true 
	\begin{enumerate}
	\item[(i)]  if $\lim\limits_{k\to\infty} \epsilon_i^k \to 0$, then $\liminf\limits_{k\to\infty} |x_i^k| > 0$; 
	\item[(ii)]  if $\lim\limits_{k\to\infty} \epsilon_i^k > 0$ meaning $\epsilon_i$ is not updated after some iteration $\tilde k$, then 
	$\lim\limits_{k\to\infty}  x_i^k = 0$. 
	\item[(iii)] if $\epsilon_i^k \equiv \epsilon_i^{\tilde k}$, then $w_i^{\tilde k} \ge \limsup\limits_{k\to\infty}  \frac{|\nabla_i f(x^k)|}{ \lambda}$
	 and $\epsilon_i^{\tilde k} \le  \liminf\limits_{k\to\infty}  \left( \frac{p \lambda}{|\nabla_i f(x^k)|}\right)^{\frac{1}{1-p}}$. 
	\end{enumerate} 
 \end{theorem}
\begin{proof}
(i) If $\epsilon^k_i \to 0$, assume by contradiction  
there exists $\Scal\subset \mathbb{N}$ such that 
$\{|x_i^k |\}_{\Scal} \to 0$.  It follows that 
$\{w_i^k\}_{\Scal} \to \infty$ since $\epsilon^k_i \to 0$. 
Therefore, there exists sufficiently large $\tilde k\in\Scal$, such that 
$w(x_i^{\tilde  k}, \epsilon_i^{\tilde  k}) > C/\lambda $ for some ${\tilde  k} \in \mathbb{N}$;  then  the optimality condition \eqref{kkt subproblem} 
	implies $x_i^{{\tilde  k} +1} = 0$. Otherwise we have 
	$ | \nabla_i Q_{\tilde  k}(x^{\tilde  k +1}) | =   \lambda w(x_i^{\tilde  k}, \epsilon_i^{\tilde  k})  > C $, 
	contradicting \cref{prop.sign}.  Monotonicity of $(\cdot)^{p-1}$ and  $0 + \epsilon_i^{\tilde k +1} \le |x_i^{\tilde k}| + \epsilon_i^{\tilde k}$ yield  
	\[w(x_i^{{\tilde  k}+1}, \epsilon_i^{{\tilde  k}+1}) =  p(0+ \epsilon_i^{{\tilde  k}+1})^{p-1}  \ge   p( |x_i^{\tilde k}| + \epsilon_i^{\tilde  k})^{p-1}  =   w(x_i^{\tilde  k}, \epsilon_i^{\tilde  k}) > C.\]
	By induction  we know that  
	$x_i^k \equiv 0$ for any $k > \tilde  k$.  Therefore, $\epsilon_i^k$ is not updated for all $k > \tilde k$ according to \eqref{update.eps}---a contradiction. 
	Therefore, $\liminf_{k\to\infty} |x_i^k| > 0$.

(ii) If $\epsilon^k_i$ is bounded away from 0, meaning it is never reduced after some iteration $\tilde k$, then 
we know $x_i^k \equiv 0$ for all 
$k > \tilde k$.

(iii)  if $\epsilon_i^k$ is not updated after $\tilde k$, it means $x_i^k \equiv 0$ for any $k > \tilde k$.  Therefore, by the optimality condition of the subproblem,  
we have  $| \nabla_iQ_k(x^{k+1})| \le  w_i^{\tilde k}$
for all $k > \tilde k$, meaning (iii) is true. 
\end{proof}

It is easy to verify  the analysis in \S\ref{sec.l1lp}   still holds true, and any limit point is still first-order optimal 
for  \eqref{prob.lpproblem}.  

\subsection{Line search} 
The satisfaction of \cref{ass.basic}(ii) by  \cref{alg.framework} 
 could be impractical since it requires the prior knowledge of 
the Lipschitz constant $L_f$ of $f$. In this subsection, we propose a line search strategy 
that still guarantee the convergence of IRL1 without knowing the value of $L_f$. 
Notice that the purpose of requiring $M$ to satisfy  \cref{ass.basic}(ii) is to guarantee 
 \eqref{q les f} in the proof of \cref{prop.sign}.  
 Alternatively, we can directly require the model $Q_k(\cdot)$ yields a new iterate $x^{k+1}$  satisfying a similar condition 
 \begin{equation}\label{ls.cond} f(x^k)  - f(x^{k+1}) \ge Q_k(x^k) - Q_k(x^{k+1}) + \gamma \|x^k-x^{k+1}\|_2^2\end{equation}
 for prescribed $\gamma > 0$. 
 To achieve this,  we can repeatedly solve the subproblem and 
   convexify the subproblem $Q_k$ by adding a proximal term 
   to $Q_k(\cdot)$ if  \eqref{ls.cond} is not satisfied, i.e., setting 
  \[  Q_k(x) \gets  \ Q_k(x) + \tfrac{\Gamma}{2}\|x-x^k\|^2_2.\]
  
\begin{algorithm}[htp]
	\caption{General  framework of  IRL1 methods with line search (IRL1-LS)}
	\label{alg.framework.ls}
	\begin{algorithmic}[1]
		\STATE{\textbf{Input:} $\mu\in(0,1), \gamma \in(0,+\infty),  \epsilon^0\in\mathbb{R}^n_{++}$ and $x^0 $.}
		\STATE{\textbf{Initialize: set $k=0$}. }
		\REPEAT 
		\STATE Reweighing:  $w(x_i^k, \epsilon_i^k)= p(|x^k_i|+\epsilon^k_i)^{p-1}.$ 
		\STATE Line search: find the smallest $\Gamma \ge 0$ such that 
		 \[ \begin{aligned} 
		 Q_k(x) \gets & \ Q_k(x) + \tfrac{\Gamma}{2}\|x-x^k\|^2_2\\
		  x^{k+1} \gets &\ \underset{x\in \mathbb{R}^n}{\text{argmin}} \quad \bigg\{Q_k(x) + \lambda \sum\limits_{i=1}^{n} w(x_i^k, \epsilon_i^k) |x_i| \bigg\}\\
		 f(x^k)  - f(x^{k+1}) \ge &\  Q_k(x^k) - Q_k(x^{k+1}) + \gamma \|x^k-x^{k+1}\|_2^2.
 \end{aligned}
 \]
		\STATE Set $\epsilon^k \le \mu \epsilon^{k-1}$, $k\gets k+1$. 	
		\UNTIL{convergence}
	\end{algorithmic} 
\end{algorithm}

   We state the IRL1 method with line search (IRL1-LS) in \cref{alg.framework.ls}, where the appropriate value of $\Gamma$ could be 
   selected as the smallest element in $\{0,  \bar \Gamma^0, \bar\Gamma^1, \bar\Gamma^2, \ldots \}$ with given $\bar\Gamma > 1$
    such that \eqref{q les f} is true. 
    This needs the solution of several additional   subproblems for each iteration. 
    Obviously, this line search procedure will terminate in finite trials since \eqref{q les f} is always satisfied 
    for any $M > L_f + 2\gamma$.   Replacing \eqref{q les f} by \eqref{ls.cond}  in the proof of \cref{prop.sign}, 
   we can obtain similar result to \cref{prop.sign} as below. 
    \begin{proposition}\label{prop.sign1} Suppose  \cref{ass.lip} and \ref{ass.basic}(i) hold. Let $\{(x^k, \epsilon^k)\}$ be the sequence generated by \cref{alg.framework.ls}.  It follows that $F(x, \epsilon)$ is monotonically decreasing over $\{(x^k,\epsilon^k)\}$ and the reduction satisfies
	\begin{align}	
	F(x^0, \epsilon^0) - F(x^k; \epsilon^k) & \ge   \gamma \sum_{i=0}^{k-1} \|x^{k+1}-x^k\|^2_2. 
	\label{eq:convergence31}
	\end{align}
	Moreover,	$\lim\limits_{k\to\infty}\|x^{k+1} - x^k\|_2 = 0$, and there exists $C> 0$ such that 
	$\|\nabla Q_k(x^{k+1})\|_\infty \le C$ for any $k \in \mathbb{N}$.
\end{proposition}
Using \cref{prop.sign1},  all the results in subsections \ref{sec.local sign}--\ref{sec.updating} 
still hold true, which 
can be verified trivially and are therefore skipped.

\section{Connection with weighted $\ell_1$ regularization}
\label{sec.connection}

 
From what we have obtained from previous sessions, we can claim that
\[ 
\boxed{\ell_p \text{ regularization is locally equivalent to a weighted $\ell_1$ regularization}}
\]
with $p\in(0,1]$.  This result is summarized below.
%
\begin{theorem}
Any point $x^*$ satisfying the first-order necessary condition  \eqref{eq:optimalcondition} of \eqref{prob.lpproblem} is also optimal for 
the weighted $\ell_1$ regularization problem 
\begin{equation}
\min_{x\in\mathbb{R}^n}  f(x) + \lambda \sum_{i=1}^n w_i |x_i|,
\end{equation}
with weights 
$w_i = p |x_i^*|^{p-1}, i\in \Ical(x^*)$ and  $w_i > |\nabla_i f(x^*)|/\lambda, i\in\Acal(x^*)$. 
\end{theorem}
 This relationship between $\ell_p$ and weighted 
$\ell_1$ regularizations is demonstrated in Figure \ref{fig.lpl1}. The contour of the weighted $\ell_1$ regularization problem is very similar to 
that of the $\ell_{0.5}$ regularization problem around the optimal solution, and they 
both attain minimum at the same point. For $\ell_1$ regularization, the contour 
is different from $\ell_{0.5}$ and it does not attain minimum at the same point as $\ell_{0.5}$.

\begin{figure}[h]
	\includegraphics[width=5.1in]{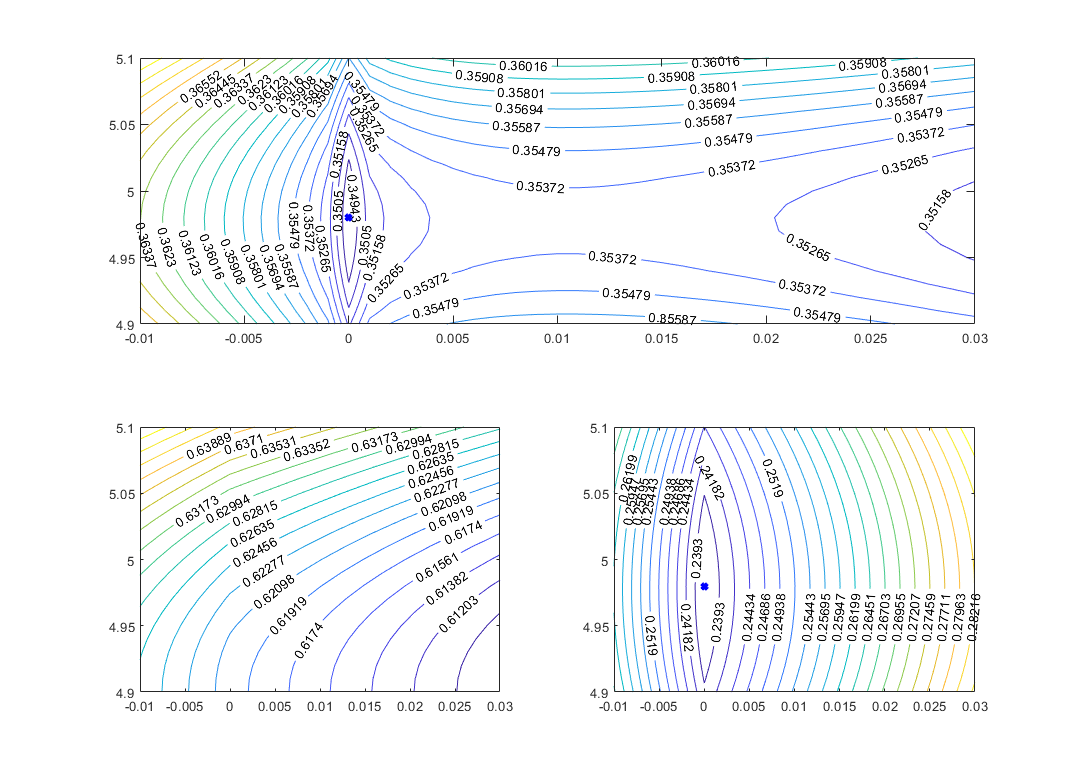}
	\centering
	\caption{Upper: the contour of $\ell_p$ ($p=0.5$) regularization problem with 
		$F(x)=(x_1-0.5)^2+(x_2-5)^2 + 0.1\|x\|_p^p$.  Bottom left: contour of 
		$F(x)=(x_1-0.5)^2+(x_2-5)^2 + 0.1\|x\|_1$.  Bottom right: contour of 
		$F(x)=(x_1-0.5)^2+(x_2-5)^2 + 0.1(20|x_1|+p(4.98)^{p-1}|x_2|)$.  The $\ell_p$ and weighted 
		$\ell_1$ problems attain optimality at the same sparse solution $(0,0.48)$.  The $\ell_1$ problem 
		does not have sparse optimal solution. }
	\label{fig.lpl1}
\end{figure}


\subsection{Maximum A Posterior (MAP)}
It is well known \cite{5256324,sokolov2019strategic} that  least squares  with $\ell_1$ regularization
 is  equivalent to finding a mode of posterior distribution 
for a linear Gaussian model with i.i.d. Laplace prior
\[ y = Ax + \theta, \text{ where } A\in\mathbb{R}^{m \times n} \text{ and }  \theta\sim N(0,\sigma^2 I_m),\]
 with $x_i \sim Laplace(0,b)$ for $i=1,\ldots,n$. The solution of MAP is as follows
\[
\hat{x} = \arg\max_x \ \log p(x|A,y, b) = \arg\min_x \tfrac{1}{2} \|Ax-y\|_2^2 + \frac{\sigma^2}{ b} \|x\|_1. 
\] 
In this case, it can be shown that   the MAP estimator for the linear model 
  with Laplace prior is the optimal solution of  
 \[\min_x \tfrac{1}{2}\|Ax-y\|_2^2 +\frac{\sigma^2}{ b} \|x\|_1.\]
 
Now suppose $x^*$ is a first-order optimal solution for the $\ell_p$ regularized least squares 
\begin{equation}\label{prob.rw1problem.lp}
\begin{aligned}
& \underset{x\in \mathbb{R}^n}{\min}
& &  \tfrac{1}{2}\|Ax-y\|_2^2+ \lambda \sum_{i=1}^{n} |x|_p^p. 
\end{aligned}
\end{equation} 
Correspondingly,  we consider  a linear Gaussian model with $\sigma^2 = \lambda$ and 
$x_i \sim Laplace(0, b_i)$, $i=1,\ldots, n$ be independent Laplace distributions with 
\[ b_i = p|x_i^*|^{1-p}, i\in\Ical(x^*)\quad \text{and}\quad b_i \le  \lambda/|\nabla_if(x^*)|, i\in\Acal(x^*).\]
Then the solution of MAP is as follows

\begin{equation}
\hat{x} = \arg\min_x\log p(x|A,y, b) =\arg\min_x \tfrac{1}{2}\|Ax-y\|_2^2 + \sigma^2\sum_{i=1}^n \tfrac{1}{b_i}|x_i|.
\end{equation}
It can be seen that the MAP estimator  for the linear model 
  with independent and non-identical Laplace prior is the optimal solution of  
\begin{equation}\label{prob.rw1problem}
\begin{aligned}
& \underset{x\in \mathbb{R}^n}{\min}
& &  \tfrac{1}{2}\|Ax-y\|_2^2+ \sigma^2 \sum_{i=1}^{n} \tfrac{1}{b_i}|x_i|. 
\end{aligned}
\end{equation} 
Therefore,   $x^*$  corresponds with 
  a MAP estimator for   the linear model 
  with independent and non-identical Laplace prior defined above. 
If we apply an IRL1 to solve  \eqref{prob.rw1problem.lp} with using the updating strategy \eqref{update.eps},  
then the   limit point of the weights $w_i^*$ yields an estimate of the $b_i$ in such a MAP model, i.e., 
$\hat b_i  =  1/w_i^*$.

\section{Numerical results}

In this section, we perform sparse signal recovery experiments (similar to \cite{yu2019iteratively,zeng2016sparse, figueiredo2007gradient,wen2018proximal}) to investigate  the behavior of  IRL1 for solving $\ell_p$ problem. 
\subsection{Experiment Setup}
We generate an $m\times n$ matrix $A$ with i.i.d. $\mathcal{N}(0,1/m)$ entries. Then set $y=Ax_{true} + e$, where the origin signal $x_{true}$ contains $K$ randomly placed $\pm1$ spikes and  $e\in \mathbb{R}^{m}$ is i.i.d.  $\mathcal{N}(0,10^{-4})$. 

We test  \cref{alg.framework.ls} for small size problems with $(m,n,K) =(256,512,64)$ and large size problems $(m,n,K)=(1024,2048,256)$. All experiments start from origin and have the same termination criteria that
\[
\max_{i\in \mathcal{I}(x)} | \nabla_if(x)+\lambda p |x_i|^{p-1} \sign( x_i) | \leq \texttt{opttol},
\]
where \texttt{opttol}$ = 10^{-6}$ is the prescribed tolerance. We also terminated if  the maximum iteration number 500 is reached. Unless otherwise mentioned, we use the following parameters to run the  experiments $\mu=0.9, \beta = 0.1, \epsilon_0 = 1, \bar{\Gamma} = 1.1, \gamma = 0.0001, p=0.5 \text{ and } \lambda = 0.05$;  $\epsilon$ is updated using the (SR) strategy. 
\subsection{Locally stable support}
 In this subsection, we run experiments to see the number of  iterations the algorithm needs to find the stable support as shown in \cref{thm.stable.support}. Let $N_S$ be the iteration number for the support to be stabilized, $N$ be the final iteration number to reach the termination criteria. Then, the ratio $N_S/N$ shows at which stage the iterate starts to obtain the stable support. 
 
 The histogram of $N_S/N$ for 1000 problems of each  size is shown in \cref{fig.new_lss}. The plot shows that the algorithm is able to reach the stable support stage in less than $50\%$ of final iterations for $98\%$ of problems. This means the stable support is identified at relatively early stage during the problem solving.
\begin{figure}[htbp]  
	\centering 
	\subfigure[The histogram about $N_s/N$ for small size problems.]{\includegraphics[width=2.75in]{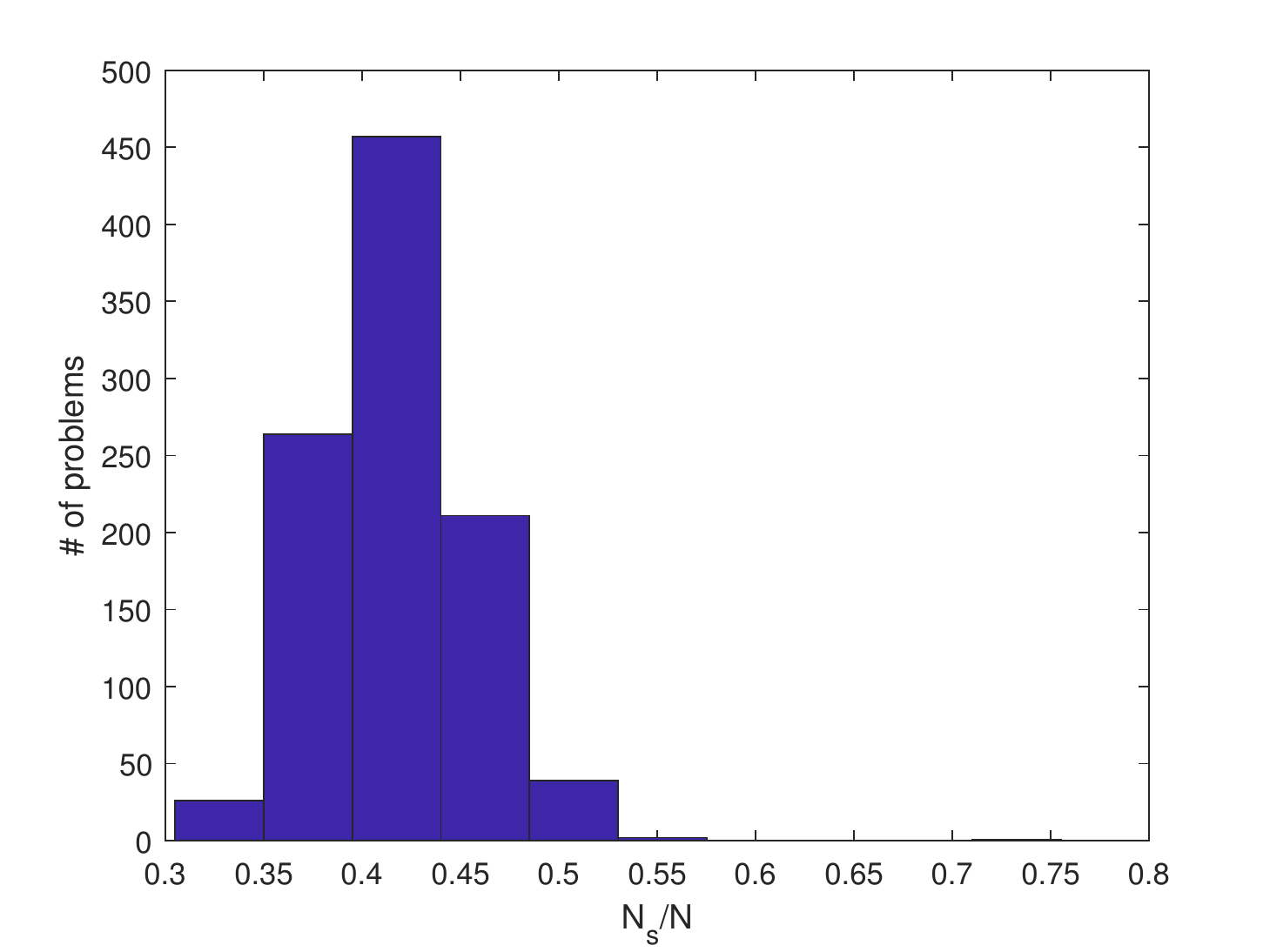}}
	\label{fig:subfig:a} 
	\subfigure[The histogram about $N_s/N$ for large size problems.]{\includegraphics[width=2.75in]{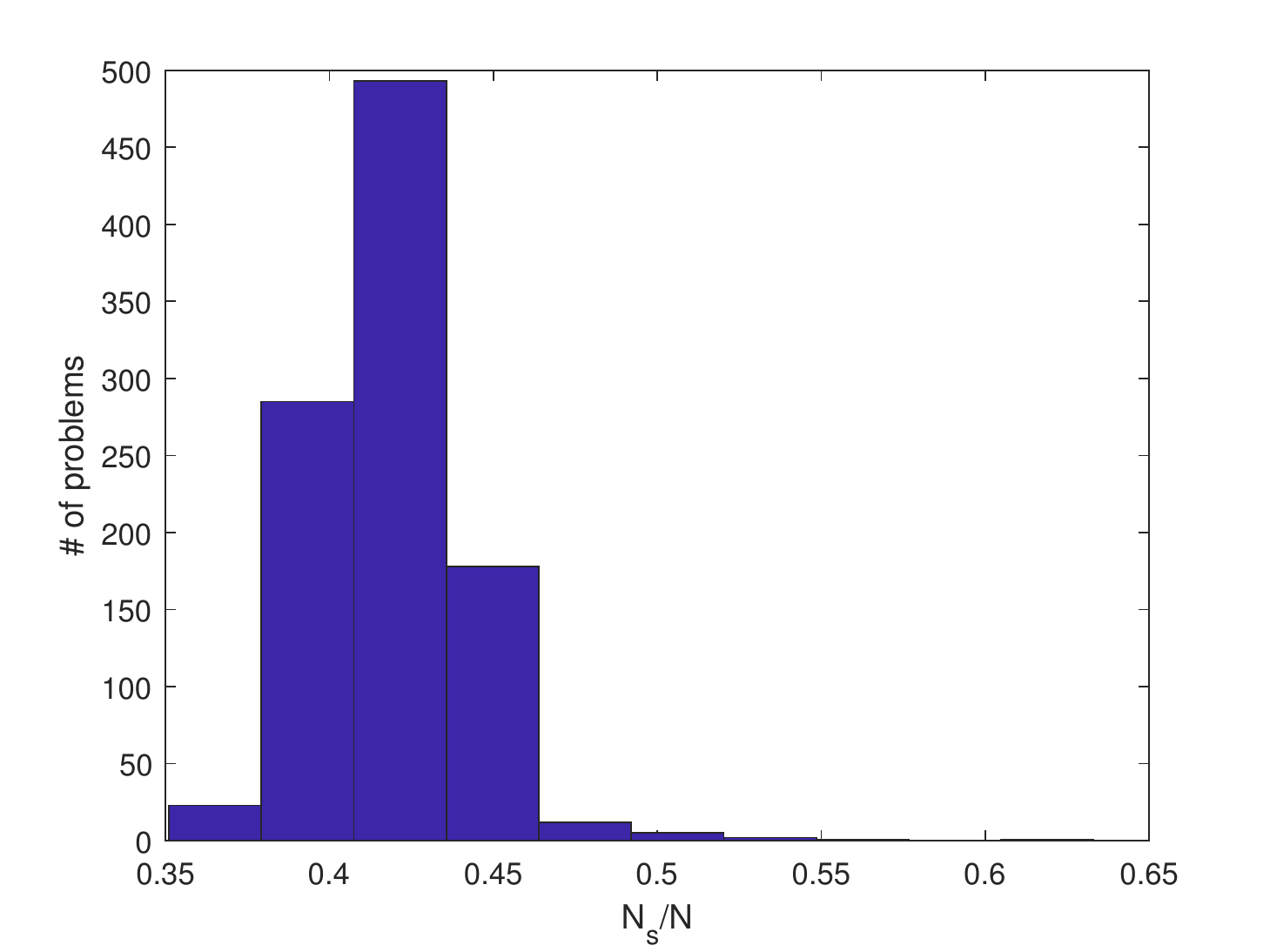}}
	\caption{The property of locally stable support for different size problems. } 
	\label{fig.new_lss}
\end{figure}
\subsection{The impact of epsilon updating strategy}
In this subsection, we test the benefits brought by our proposed  $\epsilon$ updating strategy \eqref{update.eps}.   
 We compare updating strategy $\epsilon^{k+1}=\mu \epsilon^k$ against \eqref{update.eps} updating strategy on 1000 simulated problems of each size as mentioned in experiment setup section. 

 We plot the cumulative curve of the percentage of success cases over iteration number  in \cref{fig.update_strategy1}. It clearly shows  \eqref{update.eps} updating strategy outperforms $\epsilon^{k+1}=\mu \epsilon^k$ updating. Specifically, the \eqref{update.eps} updating strategy has $90\%$ problems solved within $260$ iterations compared with around $400$ for  $\epsilon^{k+1}=\mu \epsilon^k$ updating. Besides,  the \eqref{update.eps} updating strategy solve all 1000 problems, while around $6\%$ of the problems are still unsolved for $\epsilon^{k+1}=\mu \epsilon^k$ updating.
 \begin{figure}[htbp] 
 	\centering 
 	\subfigure[The percentage of problems reaching termination criteria for small size problems]{\includegraphics[width=2.75in]{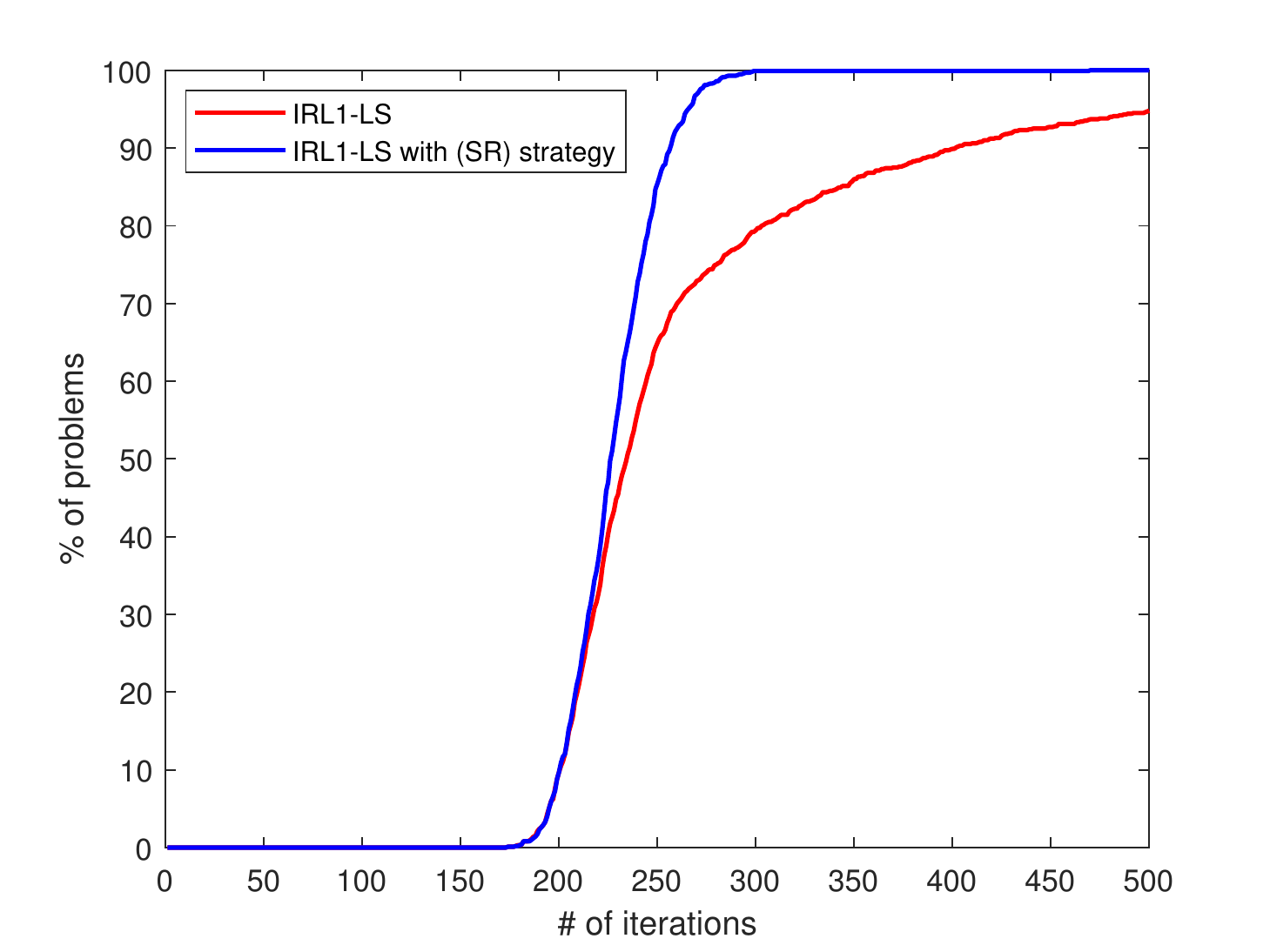}}
 	\subfigure[The percentage of problems reaching termination criteria for large size problems]{\includegraphics[width=2.75in]{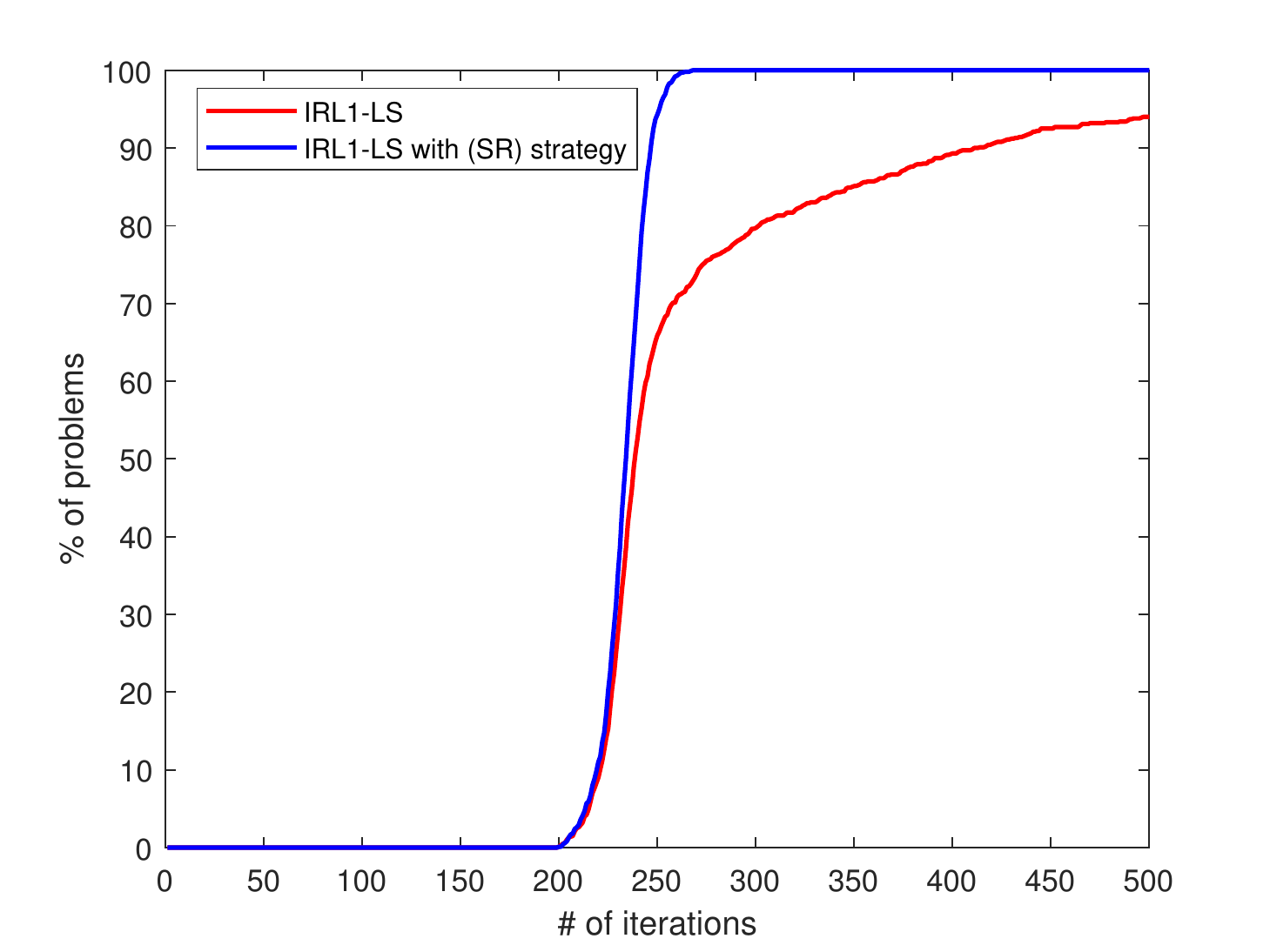}}
  	\caption{Comparison with different update strategies. } 
 	\label{fig.update_strategy1}
 \end{figure}
 
 \subsection{The impact of epsilon initialization} In this experiment, we set $\epsilon_0=$ 0.001, 0.005, 0.01 and 0.1 to see how the initialization of $\epsilon$ impact the convergence. 
 
 We plot the number of problems converging to a solution with the correct support (satisfy $\mathcal{I}(x_{true})=\mathcal{I}(x^*)$) in \cref{fig.eps0}. We make the following observation
 \begin{itemize}
 \item Larger $\epsilon_0$ has higher probability converge to the global optimal support.  
 \item If $\epsilon_0$ is too small, the algorithm will get trapped to some bad local solution. In our experiment, when $\epsilon_0=0.001$, there is no problem finding the correct support.
 \end{itemize}

 \begin{figure}[htbp] 
 	\centering 
 	\subfigure[The number of small size problems satisfying $\mathcal{I}(x_{true})=\mathcal{I}(x^*)$.]{\includegraphics[width=2.75in]{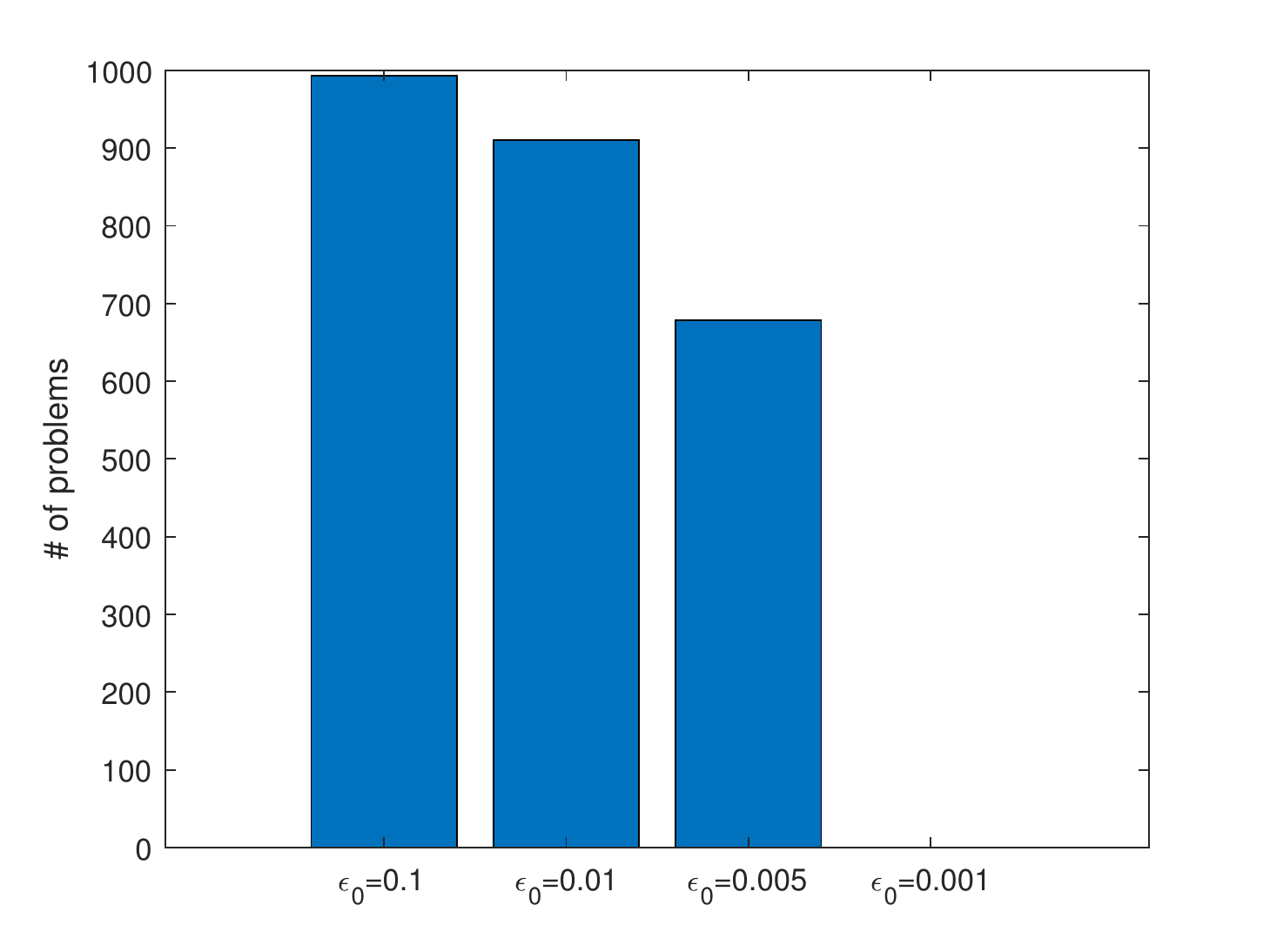}}
 	\subfigure[The number of large size problems satisfying $\mathcal{I}(x_{true})=\mathcal{I}(x^*)$.]{\includegraphics[width=2.75in]{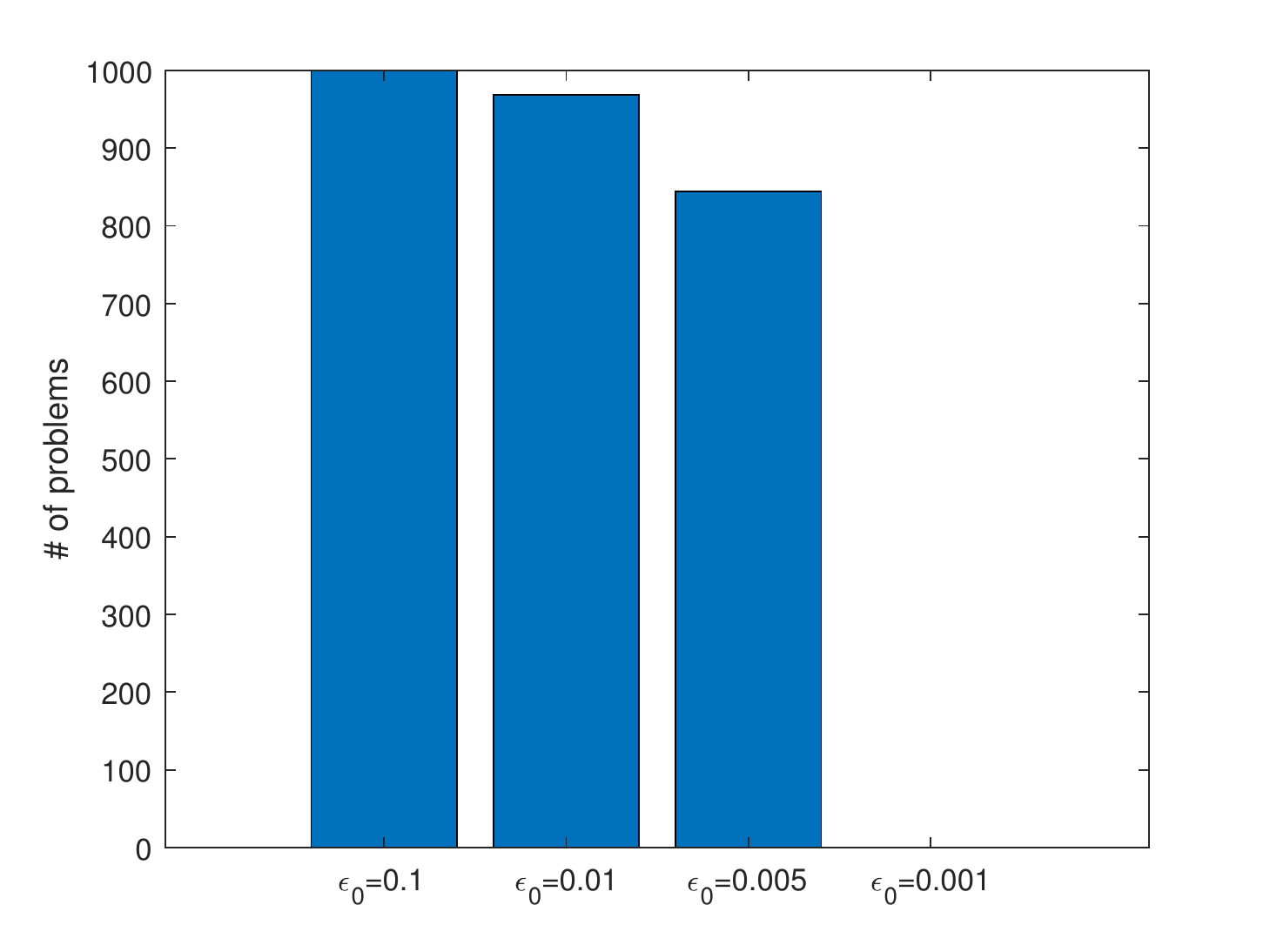}}
 	
 	\caption{The result of different epsilon initialization. It shows that the number of problems finding correct support over 1000 problems when $\epsilon_0 = 0.001, 0.005, 0.01,0.1$.  } 
 	\label{fig.eps0}
 \end{figure}


\newpage
\bibliographystyle{siamplain}
\bibliography{references}

\end{document}